\documentclass[amstex,12pt,reqno]{amsart}
\usepackage{amsmath, amssymb, amsthm}

\usepackage{mathrsfs}
\usepackage[all]{xy}
\usepackage{tikz}
\usetikzlibrary{intersections,calc,arrows.meta,patterns}

\newtheorem{df}{Definition}
\newtheorem{prop}{Proposition}
\newtheorem{theo}{Theorem}
\newtheorem{lem}{Lemma}
\newtheorem{cor}{Corollary}

\title
[Dirichlet spaces over chord-arc domains]
{Dirichlet spaces over chord-arc domains}

\author[H. Wei]{Huaying Wei} 
\address{Department of Mathematics and Statistics, Jiangsu Normal University \endgraf Xuzhou 221116, PR China} 
\email{hywei@jsnu.edu.cn} 

\author[M. Zinsmeister]{Michel Zinsmeister}
\address{Institut Denis Poisson, Universit\'e d' Orl\'eans, 
Orl\'eans, 45067, France}
\email{zins@univ-orleans.fr}

\makeatletter
\@namedef{subjclassname@2020}{%
\textup{2020} Mathematics Subject Classification}
\makeatother

\thanks{Research supported by 
 the National Natural Science Foundation of China (Grant No. 12271218).}

\begin{document}

\begin{abstract}
If $U$ is a $C^{\infty}$ function with compact support in the plane, we let $u$ be its restriction to the unit circle $\mathbb{S}$, and denote by $U_i,\,U_e$ the harmonic extensions of $u$ respectively in the interior and the exterior of $\mathbb S$ on the Riemann sphere. About a hundred years ago, Douglas \cite{dou} has shown that
\begin{align*}
    \iint_{\mathbb{D}}|\nabla U_i|^2(z)dxdy&= \iint_{\bar{\mathbb{C}}\backslash\bar{\mathbb{D}}}|\nabla U_e|^2(z)dxdy\\
    &= \frac{1}{2\pi}\iint_{\mathbb S\times\mathbb S}\left|\frac{u(z_1)-u(z_2)}{z_1-z_2}\right|^2|dz_1||dz_2|,
\end{align*}
thus giving three ways to express the Dirichlet norm of $u$. On a rectifiable Jordan curve $\Gamma$ we have obvious analogues of these three expressions, which will of course not be equal in general. The main goal of this paper is to show that these $3$ (semi-)norms are equivalent if and only if $\Gamma$ is a chord-arc curve.
\end{abstract}

\maketitle

\section{Introduction}
If $\Omega$ is a domain of the Riemann sphere, the Dirichlet space $\mathcal{D}(\Omega)$ is the space of harmonic functions $U:\Omega\mapsto \mathbb{C}$ such that
$$\iint_\Omega|\nabla U|^2(z)dxdy<\infty.$$
The Dirichlet space over the unit disk $\mathbb D$ is, together with the Hardy and Bergman spaces, one of three most classical Hilbert spaces in the unit disk. The key references are survey articles \cite{arc,Ros} and the book \cite{EKMR}. 
It plays an important role in domains as distinct as Minimal Surfaces, Operator Theory and Teichm\"uller Theory (e.g.\cite{nag}). Concerning the latter, there has been recently a regain of interest with the explosion of results about Weil-Petersson curves; that is the connected component of the identity in the universal Teichm\"uller space viewed as a complex Hilbert manifold; see   \cite{bis, she, tak, wan}.

In the case of unit disk, we will see in the next section that functions $U_i$ in $\mathcal{D}(\mathbb{D})$ have well-defined limits $u$ on the unit circle $\mathbb{S}$ that moreover characterize $U_i$. We may thus consider the Dirichlet space $\mathcal{D}(\mathbb{D})$ as a quotient space of a space of functions defined on  $\mathbb S$ modulo constants. Moreover, using the reflection $z\mapsto 1/\bar z$ we have that this latter space coincides with the space of boundary values of functions $U_e$ in $\mathcal{D}(\mathbb{D}_e)$ , where $\mathbb{D}_e$ is the unbounded component of $\bar{\mathbb{C}}\backslash \mathbb{S}$. A deep theorem of Douglas \cite{dou} asserts that this space of functions coincides with the space $H^{1/2}(\mathbb{S})$.  Its norm may thus be expressed in three equal ways as follows:
\begin{align*}
 \|u\|_i^2 :=&\frac{1}{2\pi}\iint_{\mathbb D} |\nabla U_i|^2(z)dxdy\\
 = &\|u\|_e^2 :=\frac{1}{2\pi}\iint_{\mathbb{D}_e} |\nabla U_e|^2(z)dxdy\\
= &\|u\|_{H^{1/2}(\mathbb S)}^2 := \iint_{\mathbb S \times \mathbb S}  \Big|\frac{u(z_1) - u(z_2)}{z_1 - z_2}\Big|^2 \frac{|dz_1|}{2\pi}\frac{|dz_2|}{2\pi},
 \end{align*}
where  $\|u\|_i^2$ is called the interior energy of $u$, while $\|u\|_e^2$ is the exterior energy.   $\|\cdot\|_{H^{1/2}(\mathbb S)}$ is called the Douglas norm, and the equality of the first and third integrals is the Douglas formula, 
as introduced by Douglas in his solution of the Plateau problem (\cite{ah2, dou}). The formula inspired important developments in the theory of Dirichlet forms; see \cite{fuk}.

Inspired by Problem 38 of \cite{arc} which consists in developing a theory of Dirichlet spaces in planar domains, our aim in this paper is to start 
investigating Dirichlet spaces over planar domains and more specifically to find the right class of Jordan domains for which these three norms make sense and are equivalent (we cannot of course expect the equality in general). Since the analogue of the Douglas norm makes sense only for rectifiable curves we will , from now on, restrict our study to rectifiable Jordan curves. For such a curve $\Gamma$ we thus define the space (modulo constants) $H^{1/2}(\Gamma)$ of mesurable functions $u:\Gamma\to \mathbb{C}$ such that
$$\|u\|_{H^{1/2}(\Gamma)}^2 := \frac{1}{4\pi^2} \iint_{\Gamma\times \Gamma}  \Big|\frac{u(z_1) - u(z_2)}{z_1 - z_2}\Big|^2 |dz_1||dz_2|<\infty.$$
About the other two norms $\|\cdot\|_i,\,\|\cdot\|_e$, we will show that it  happens to be true in the good case that the spaces of boundary functions  of $\mathcal{D}(\Omega)$ and $\mathcal{D}(\Omega_e)$ coincide, where $\Omega$ and $\Omega_e$ denote the bounded and unbounded components of the complement of $\Gamma$, respectively. In order to make an explanation on this claim at this stage, 
we consider instead the space $E(\Gamma)$ of restrictions to $\Gamma $ of $C^{\infty}$ functions with compact support in the whole plane, and will see that the norms $\|\cdot\|_i,\,\|\cdot\|_e$ are equivalent on $E(\Gamma)$ for a rectifiable Jordan curve $\Gamma$ if and only if it is a quasicircle. We may define, for such a curve, $\mathcal{H}(\Gamma)$ as the completion of $E(\Gamma)$ with respect to one of these two  equivalent norms (but this is not the order in which we will derive things).

The main result of this paper will be that  $H^{1/2}(\Gamma)=\mathcal{H}(\Gamma)$ if and only if $\Gamma$ is a chord-arc curve, a rectifiable quasicircle with the property of Ahlfors-regularity. 

The paper is structured as follows: In Sect. 2 we give the precise definition of the Dirichlet space over any Jordan domain. 
In Sect. 3 we study the case of quasicircles and define precisely the ``two-sided" space $\mathcal{H}(\Gamma)$. 
Sect. 4 is devoted to prove that $\mathcal{H}(\Gamma)=H^{1/2}(\Gamma)$ when $\Gamma$ is a chord-arc curve. 
In Sect. 5 we consider the sharpness of the result in Sect. 4: we prove more precisely that if $\Gamma$ is a rectifiable quasicircle such that $\mathcal{H}(\Gamma)\subset H^{1/2}(\Gamma)$ then $\Gamma$ must be chord-arc and the two spaces are equal. On the other hand we construct a rectifiable quasicircle $\Gamma$ that is not chord-arc but such that $ H^{1/2}(\Gamma)\subset\mathcal{H}(\Gamma)$.

\section{Dirichlet spaces over a Jordan domain}
Recall that if  $\Omega$ is a bounded Jordan domain of the complex plane $\mathbb C$, the Dirichlet space $\mathcal D(\Omega)$ is the space of harmonic functions $F$ on $\Omega$ with finite Dirichlet energy $D(F)$, where the energy of any $C^1$ map on $\Omega$ is defined as the $L^2(\Omega)$-norm of the gradient vector $\nabla F(w) = (F_u, F_v) $. Precisely, 
\begin{equation}\label{DD}
    D(F) :=\frac{1}{2\pi}\iint_{\Omega} |\nabla F|^2(w)dudv<+\infty.
\end{equation}
The complex notation is much more convenient. Let us first note 
$F_{\bar w} = (F_u +i F_v)/2$ and $F_w = (F_u - i F_v)/2$.  This gives
$D(F) = 1/\pi\iint_{\Omega} (|F_w|^2 + |F_{\bar w}|^2) dudv$.  The space $\mathcal D(\Omega)$ is a Hilbert space modulo constant functions. Let $\varphi$ map another bounded Jordan domain  $\Omega'$ conformally onto $\Omega$. Using 
$$
|G_z|^2 + |G_{\bar z}|^2 = (|F_w|^2 + |F_{\bar w}|^2)\circ \varphi(z) |\varphi'(z)|^2, 
$$
we see that $F\mapsto G := F\circ\varphi$ is a bijective isometry between $\mathcal D(\Omega)$ and  $\mathcal D(\Omega')$.  Similarly, an anti-conformal mapping $\varphi(\bar z)$ also induces the invariance of Dirichlet energies. 

In the case of classical Dirichlet space $\mathcal D(\mathbb D)$, one may make the theory precise  by the use of Fourier series. For a real-valued function $u \in \mathcal D(\mathbb D)$, let $v$ be the unique harmonic conjugate function of $u$ in $\mathbb D$ with the requirement $v(0) = 0$ so that $\Phi = u + iv$ is holomorphic.  An easy calculation involving Parseval formula leads to, writing $\Phi(z) = \sum_{n=0}^{\infty}a_n z^n$, 
$$
D(\Phi) = \sum_{n=1}^{\infty} n|a_n|^2.
$$
By Cauchy-Riemann equations, $D(u) = D(v)$, and then, $D(u) = \frac{1}{2}\sum_{n=1}^{\infty} n|a_n|^2$.  In particular, $ \sum_{n=1}^{\infty} |a_n|^2<\infty$, which means that  $\Phi\in H^2(\mathbb D)$, the Hardy space of analytic functions on $\mathbb D$, and that the function $u$ belongs to the Hardy space of harmonic functions $h^2(\mathbb D)$. As a consequence, $u$ has angular limits almost everywhere on the unit circle $\mathbb S$. 

Suppose now that $\Omega$ is a domain bounded by a rectifiable Jordan curve $\Gamma$. Let $\varphi$ map  $\mathbb D$ conformally onto $\Omega$. Then $\varphi$ extends to a homeomorphism of the closures $\mathbb D \cup \mathbb S$ and $\Omega\cup\Gamma$. Using F. and M. Riesz theorem  for the curve $\Gamma$ being rectifiable (see [18]), we have that a subset of $\mathbb S$ has measure zero if and only if its image under $\varphi$ has length zero, and it also makes sense to speak of a tangential direction almost everywhere on $\Gamma$. Furthermore, the mapping $\varphi$ preserves angles at almost every boundary point on $\Gamma$, see \cite{dur, pom} for details. Consequently, for any $F \in \mathcal D(\Omega)$, since  $G := F\circ\varphi \in \mathcal D(\mathbb D)$ one may say that $F$ has angular limits $f(w)$ almost everywhere on $\Gamma$ such that $f\circ\varphi \in L^2(\mathbb S)$. 
So $F$ can be recovered from its boundary function by the ``Poisson integral" of $f$, in the sense that 
\begin{equation}\label{Ff}
    F = P(f\circ \varphi)\circ\varphi^{-1}
\end{equation}
where $P$ stands for the classical Poisson integral in the unit disk $\mathbb D$. 
If $f_1$ and $f_2$ are boundary functions of $F_1$ and $F_2$ in $\mathcal{D}(\Omega)$, respectively, and $f_1 = f_2$ almost everywhere on $\Gamma$, we then have $F_1 = F_2$ by \eqref{Ff}. Hence, we may say that $f_1 = f_2$ if they are equal except possibly on a subset with length zero.
The one to one correspondence $F\leftrightarrow f$ allows us to view the Dirichlet space on $\Omega$ as a space of functions defined on $\Gamma$. 
We denote it by $\mathcal{H}(\Gamma, \Omega)$.

Since the function $F$ in \eqref{Ff} is the solution to Poisson's equation $\Delta F = 0$ in the domain $\Omega$ with boundary condition $F|_{\Gamma} = f$, Dirichlet's principle says that $F$ can be obtained as the minimizer of Dirichlet energies $D(V)$ 
amongst all $C^1$ extensions $V$ to $\Omega$ of $f$. We call $\Vert f \Vert_i^2 := D(F)$ the interior energy of $f$. Clearly, $\mathcal{H}(\Gamma, \Omega)$ is the function space consisting of all $f$ with a finite semi-norm $\Vert f \Vert_i$. 

Let us make one remark on Dirichlet's principle: The requirement for $C^1$ extensions can be relaxed. Precisely,  
let $\dot W^{1,2}(\Omega)$ be the homogeneous Sobolev space of locally integrable functions on any domain $\Omega$ of the Riemann sphere with $L^2$-integrable gradient taken in the sense of distributions, equipped with the natural $L^2(\Omega)$-norm of the gradient vector $\nabla F$ as in \eqref{DD}. It has an important subspace $\dot W^{1,2}_0(\Omega)$ defined as the closure in $\dot W^{1,2}(\Omega)$ of $C_0^\infty(\Omega)$, the space of $C^\infty$ functions with compact support included in $\Omega$. By the Meyers-Serrin theorem \cite{MS}, the space $C^\infty(\Omega)\cap \dot W^{1,2}(\Omega)$ is dense in $\dot W^{1,2}(\Omega)$. With this setting, a simple approximation argument shows that
\begin{equation}\label{Diri}
   D(F) = \min\{D(V): V \;\text{ranges over all  functions in}\; \dot W^{1,2}(\Omega) \;\text{with}\;  V-F \in \dot W_0^{1,2}(\Omega)\}. 
\end{equation}

By the Jordan curve theorem the complement of the boundary Jordan curve $\Gamma$ of $\Omega$ has two components: one is $\Omega=\Omega_i$, the so-called interior of $\Gamma$, and the second is $\Omega_e$, the exterior of  $\Gamma$. In order to define a Dirichlet space over  $\Omega_e$, we may first assume without loss of generality that $0 \in \Omega$, and we consider the reflection $\iota(z)=\frac{1}{\bar z}.$ It maps $\Omega_e$ onto a  Jordan domain $\tilde\Omega = \tilde\Omega_i$ bounded by a bounded Jordan curve $\tilde\Gamma$.  The image of $\Omega$ is  $\tilde\Omega_e$, the exterior of $\tilde\Gamma$. One may see that 
the point $0 = \iota(\infty)$ is an interior point of $\tilde\Omega$.  Define $\mathcal{D}(\Omega_e)$ as the set 
$\{F := \tilde F \circ\iota,\,\tilde{F}\in \mathcal{D}(\tilde{\Omega})\}$. 
Then the Dirichlet energy of $F$ over $\Omega_e$ can be written as 
$$ D(F)=\frac{1}{2\pi}\iint_{\Omega_e\backslash{\{\infty\}}}|\nabla F|^2(w)dudv.$$
That is, the point  $\infty$ can be removed from the domain of integration without changing the convergence properties or the value of integral. This is justified as follows:  if $F \in \mathcal{D}(\Omega_e)$ then for any $\epsilon > 0$ there is an $R > 0$ such that 
$$
\iint_{|w| > R} |\nabla F|^2(w) dudv = \iint_{|z| \leq  1/R} |\nabla \tilde F|^2(z) dxdy < \epsilon.
$$
Let $\tilde\varphi$ map $\mathbb D = \iota(\mathbb D_e)$ conformally onto $\tilde\Omega = \iota(\Omega_e)$.  
The one to one correspondence $\tilde F = P(\tilde f\circ \tilde\varphi)\circ \tilde\varphi^{-1} \leftrightarrow \tilde f$ between the elements of $\mathcal{D}(\tilde{\Omega})$ and $\mathcal{H}(\tilde\Gamma, \tilde\Omega)$ leads to the one to one correspondence $F = \tilde F\circ \iota \leftrightarrow f = \tilde f \circ \iota$ between the elements of $\mathcal{D}(\Omega_e)$ and $\mathcal{H}(\Gamma, \Omega_e)$. Here, $\mathcal{H}(\Gamma, \Omega_e)$ is a function space consisting of  all boundary functions $f$ of elements $F$ of $\mathcal{D}(\Omega_e)$ assigned a semi-norm $\Vert f \Vert_e$,  where   $\Vert f \Vert_e^2 := D(F)$. We call $\Vert f \Vert_e^2$ the exterior energy of $f$. 
 
For the simple case $\Omega = \mathbb D$, we notice that $\tilde\Omega = \mathbb D$ by $\iota(z) = z$ on $\mathbb S$. Then,  the identity operator  
$\mathcal{H}(\mathbb S, \mathbb D) \to \mathcal{H}(\mathbb S, \mathbb D_e)$ is an isometric isomorphism with respect to $\Vert \cdot \Vert_{i}$ and $\Vert \cdot \Vert_{e}$.
In the next section, we  investigate the Jordan domain to which this property may be generalized.

\section{Dirichlet spaces over  quasidisks}
In this section we look for a sufficient and necessary condition on the rectifiable Jordan curve $\Gamma$ so that 
the interior and exterior energies on $\Gamma$ are equivalent. Precisely, we will show  the following 
\begin{theo}\label{quasi}
   With the above notation,  the identity operator $\mathcal H(\Gamma, \Omega) \to \mathcal H(\Gamma, \Omega_e)$ is a bounded isomorphism with respect to $\Vert \cdot \Vert_{i}$ and $\Vert \cdot \Vert_{e}$ if and only if $\Gamma$ is a quasicircle.
\end{theo}
Before starting the proof of Theorem \ref{quasi}, we recall some preliminary facts about 
 quasicircles and  quasisymmetric mappings; see \cite{ahl} for additional background. 
A quasicircle is the image of a circle under a quasiconformal mapping of the complex plane $\mathbb C$, and the inner domain of a quasicircle is called a quasidisk. Here, by  quasiconformal mapping $f$ of $\mathbb C$ we mean a homeomorphism $f$ whose  gradient in the sense of distribution belongs to $L^2_{loc}(\mathbb C)$ and satisfies 
$$f_{\bar z}=\mu(z)f_{z}$$
for an essentially uniformly bounded function $\mu\in L^\infty(\mathbb C)$ bounded by some constant $k < 1$. Here, $f$ may also be called $k$-quasiconformal to specify the constant $k$. 

A sense-preserving homeomorphism $h$ of $\mathbb S$ is called a conformal welding of the Jordan curve $\Gamma$ if $h = \psi^{-1}\circ\varphi$ where $\varphi$ and $\psi$ are conformal maps from $\mathbb D$ onto $\Omega$ and from $\mathbb D_e$ onto $\Omega_e$, respectively. So there are many weldings of $\Gamma$ but they differ from each other by compositions with M\"obius transformations of $\mathbb S$. In particular, the conformal welding of a quasicircle is exactly a quasisymmetric homeomorphism of $\mathbb S$, and
the conformal maps $\varphi$ and $\psi$ for a quasicircle  can be  extended to $\mathbb C$ quasiconformally.  
Saying a homeomorphism $h$ of $\mathbb S$  is quasisymmetric means that there exists a  constant $C_1 > 0$   such that
$$
C_1^{-1} \leq \frac{|h(e^{i(\theta + \alpha)}) - h(e^{i\theta} )|}{|h(e^{i\theta}) - h(e^{i(\theta - \alpha)} )|} \leq C_1
$$
for all $\theta \in \mathbb R$ and $-\pi/2 < \alpha \leq \pi/2$. Here, the optimal constant $C_1$ is called the quasisymmetry constant of $h$. A sense-preserving homeomorphism $h$ of $\mathbb S$ is quasisymmetric if and only if $h$ preserves the modules of quadrilaterals quasi-invariantly, namely, there exists a constant $C_2 > 0$ such that for any quadrilateral $Q$ it holds that 
$$C_2^{-1} m(Q) \leq m(h(Q)) \leq C_2 m(Q).$$ 
Moreover,  the optimal constant $C_2$ and the quasisymmetry constant $C_1$ of $h$ only depend on each other. Here, by quadrilateral $Q$ we mean the unit disk $\mathbb D$ together with a pair of disjoint closed arcs  on the boundary 
$\mathbb S$.  It is a well-known fact that for a quadrilateral $Q$ with two disjoint closed arcs $\alpha_1$, $\beta_1$ on $\mathbb S$, its   module $m(Q)$ multiplied by $\frac{1}{2\pi}$ is exactly the minimum of Dirichlet energies $D(P(f))$ of harmonic functions $P(f)$ on $\mathbb D$ ranging over all boundary values $f$ with $f = 0$ on $\alpha_1$ and $f = 1$ on $\beta_1$ (see \cite{BA}). Then, by the definition of $\Vert f \Vert_i^2$,  we have $\frac{1}{2\pi}m(Q) = \min\Vert f \Vert_i^2$.

The above concept of quasisymmetry of $\mathbb S$ onto $\mathbb S$ was introduced by Beurling and Ahlfors \cite{BA}, and later formulated for general metric spaces by Tukia and V\"ais\"al\"a \cite{TV}. For our purpose, we only need to consider the quasisymmetric mapping from a curve $\Gamma_1$ onto the other $\Gamma_2$. Let $h: \Gamma_1\to\Gamma_2$ be a sense-preserving mapping. Let $\eta: [0, +\infty) \to [0, +\infty)$ be an increasing homeomorphism with $\lim_{t \to +\infty}\eta(t) = +\infty$. We say that $h$ is $\eta$-quasisymmetric if for each triple $z_0$, $z_1$, $z_2 \in \Gamma_1$ we have
$$
\frac{|h(z_0) - h(z_1)|}{|h(z_0) - h(z_2)|} \leq \eta \bigg( \frac{|z_0 - z_1|}{|z_0 - z_2|} \bigg).
$$
The more general quasisymmetric mapping  on an open subset of $\mathbb C$ can be defined in the same way. 
It is known that if $h: \mathbb C \to \mathbb C$ is a $k$-quasiconformal homeomorphism of $\mathbb C$ then $h$ is $\eta$-quasisymmetric where $\eta$ depends only on $k$. Conversely,  if $h: D \to \mathbb C$ is an $\eta$-quasisymmetric mapping on a domain $D$ then it is quasiconformal (see e.g. Chapter 3 of \cite{AIM}).

Concerning the quasisymmetric homeomorphism of $\mathbb S$ onto $\mathbb S$, the following result of Nag-Sullivan \cite{nag} is well-known. 
\begin{prop}\label{compo}
    A sense-preserving homeomorphism $h$ of $\mathbb S$ is quasisymmetric if and only if the composition operator $V_h:g\mapsto g\circ h$ gives an isomorphism of $H^{1/2}(\mathbb S)$, that is, $V_h$ and $(V_h)^{-1}$  (or $V_{h^{-1}}$) are bounded linear operators. 
    Here, the operator norm $\Vert V_h \Vert$ depends only on the quasisymmetry constant of $h$.  
\end{prop}
\noindent We will also  give its generalization to the quasisymmetry of $\mathbb S$ onto a curve $\Gamma$ as three corollaries of main theorems in the remainder, which can be summarized by 
\begin{prop}
    Let $\Gamma$ be a rectifiable quasicircle. Let $h$ be a sense-preserving homeomorphism of $\mathbb S$ onto $\Gamma$. The composition operator $V_{h} : g \mapsto g\circ h$ is defined on $H^{1/2}(\Gamma)$. Consider the following four statements: 
        \begin{itemize}
        \item[{\rm(a)}] $\Gamma$ is  chord-arc,
        \item[{\rm(b)}] $h$ is quasisymmetric,
        \item[{\rm(c1)}] $(V_{h})^{-1}$ is a bounded linear operator from $H^{1/2}(\mathbb S)$ into $H^{1/2}(\Gamma)$,
        \item[{\rm(c2)}] $V_{h}$ is a bounded linear operator from $H^{1/2}(\Gamma)$ into $H^{1/2}(\mathbb S)$.
         \end{itemize}
         If any two of the above three statements $\rm(a)$,$\rm(b)$,and $\rm(c1)$ hold, then the third one holds true, while $\rm(a)$ and  $\rm(b)$ imply $\rm(c2)$, $\rm(a)$ and  $\rm(c2)$ imply $\rm(b)$, but $\rm(b)$ and  $\rm(c2)$ does not necessarily imply $\rm(a)$.
          \end{prop}
\noindent Actually one can view $\rm(a,c1)\Rightarrow \rm(b)$ and $\rm(a,c2)\Rightarrow \rm(b)$ are due to Proposition \ref{compo}. The complete version of $\rm(a,b)\Rightarrow \rm(c1,c2)$ is Corollary \ref{Cor1}. 
$\rm(b,c1)\Rightarrow \rm(a)$ and $\rm(b,c2)\nRightarrow \rm(a)$ are just Corollary \ref{Cor2} and  Corollary \ref{Cor3}, respectively.

We are now ready to prove Theorem \ref{quasi}.
\begin{proof}
If $\Gamma$ is a quasicircle, then there is a bi-Lipschitz map $\phi$ that fixes $\Gamma$ pointwise and exchanges the two complementary components $\Omega$ and $\Omega_e$ of $\Gamma$, see \cite{ahl}. For any  $f \in \mathcal{H}(\Gamma, \Omega)$, recall that $F = P(f\circ\varphi)\circ\varphi^{-1}$ extends $f$ to $\Omega$, minimizing  Dirichlet energies for the boundary function $f$. 
Then $F\circ\phi$ extends $f$ to $\Omega_e$ and 
$$
\frac{1}{2\pi}\iint_{\Omega_e\backslash \{\infty\}} |\nabla (F\circ\phi)|^2(w)dudv \approx \frac{1}{2\pi}\iint_{\Omega} |\nabla F|^2(z)dxdy = \Vert f \Vert_i^2
$$
where the implied constants depend only on the bi-Lipschitz constant of $\phi$. We conclude by the definition of $\Vert f \Vert_e^2$ and \eqref{Diri} that 
$$
\Vert f \Vert_e^2 \leq \frac{1}{2\pi}\iint_{\Omega_e\backslash \{\infty\}} |\nabla (F\circ\phi)|^2(w)dudv 
$$
and thus  $\Vert f \Vert_e \lesssim \Vert f \Vert_i$. The roles of $\Omega$ and $\Omega_e$ can be switched in the above argument, so that we also have  $\Vert f \Vert_i \lesssim \Vert f \Vert_e$ for any $f \in \mathcal{H}(\Gamma, \Omega_e)$. This gives a  proof of sufficiency of Theorem \ref{quasi}.

Conversely, let $\varphi$ map $\mathbb D$ conformally onto $\Omega$ with the normalizations $\varphi(0) = 0$ and $\varphi'(0) = 1$, and  $\tilde\varphi$ map $\mathbb D = \iota(\mathbb D_e)$ conformally onto $\tilde\Omega = \iota(\Omega_e)$ with  $\tilde\varphi(0) = 0$,  $\tilde\varphi'(0) = 1$.  We denote $\iota\circ\tilde\varphi\circ\iota$ by $\psi$ that  maps $\mathbb D_e$ onto $\Omega_e$. Then, these three maps are homeomorphisms of the closures. For any $f \in \mathcal H(\Gamma, \Omega)$, we see that $f\circ\iota$ denoted by  $\tilde f$ is defined on $\tilde\Gamma$. See the following commutative diagram for a picturesque description of these maps.  According to it we see  
\begin{align*}
    \tilde f\circ\tilde{\varphi}&=(f\circ\iota)\circ(\iota\circ\psi\circ\iota)\\
    &=f\circ\psi\circ\iota
     =f\circ\varphi\circ(\varphi^{-1}\circ\psi)\circ\iota \;\\
     & =f\circ\varphi\circ(\varphi^{-1}\circ\psi).
\end{align*}
Set $\psi^{-1}\circ\varphi = h$ so that 
\begin{equation}\label{fh}
    f\circ\varphi = \tilde f\circ\tilde{\varphi}\circ h.
\end{equation}
Denote the quadrilateral with two disjoint closed arcs $\alpha_1$, $\beta_1$ on $\mathbb S$ by $Q$, and denote the image of $Q$ under $h$ by $Q'$, the quadrilateral with two disjoint closed arcs $\alpha_2 := h(\alpha_1)$ and $\beta_2 := h(\beta_1)$. 
 
\begin{figure}
    \centering
    \includegraphics[width=12cm]{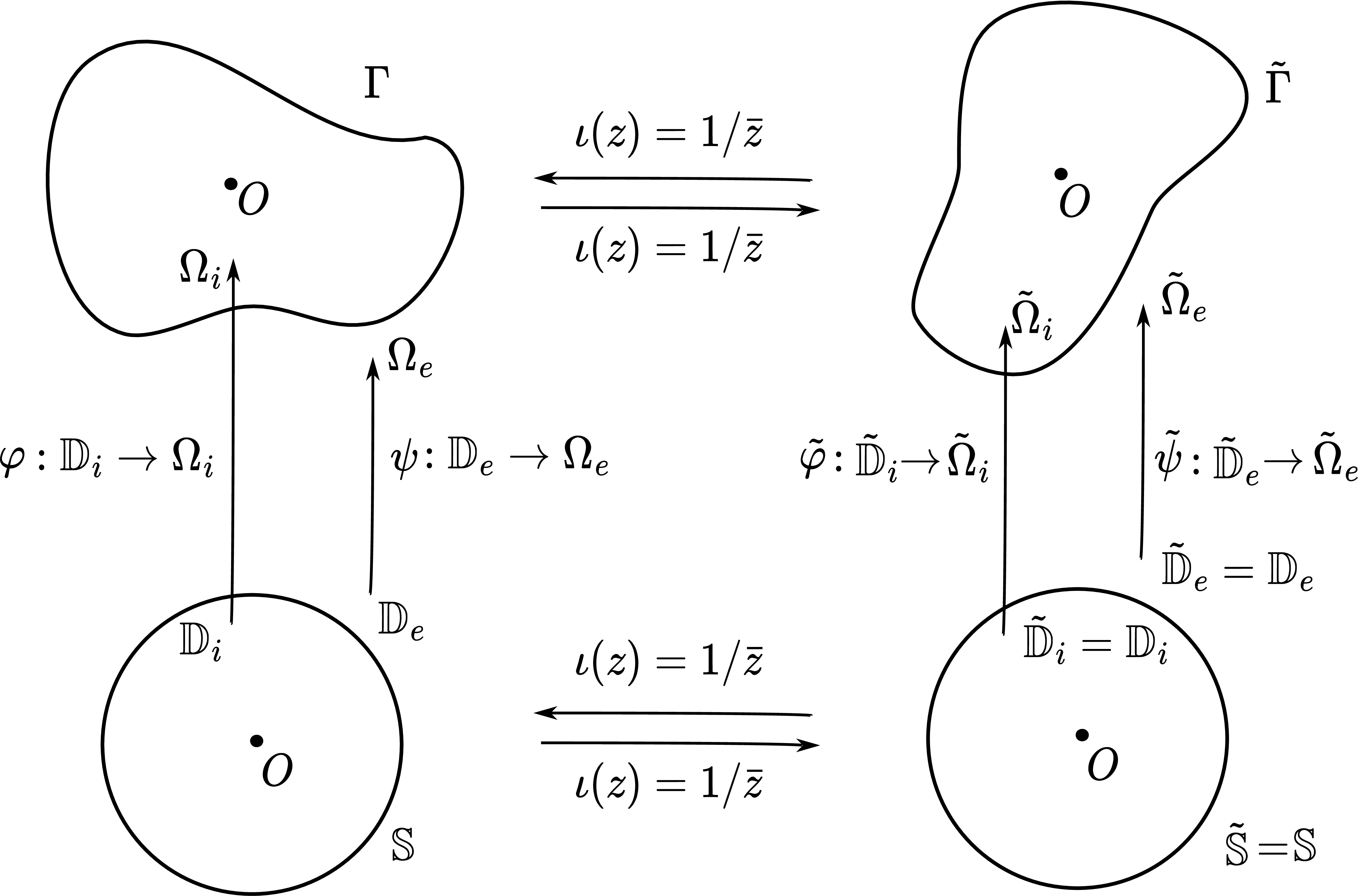}
\end{figure}

Take $f \in \mathcal{H}(\Gamma, \Omega)$ so that $f\circ\varphi$ is given only on part of $\mathbb S$: $f\circ\varphi = 0$ on $\alpha_1$ and $f\circ\varphi = 1$ on $\beta_1$. By \eqref{fh} we see that $\tilde f\circ\tilde{\varphi} = 0$ on $\alpha_2$ and $\tilde f\circ\tilde{\varphi} = 1$ on $\beta_2$. Then, we have $\frac{1}{2\pi}m(Q) = \min\Vert f\circ\varphi \Vert_i^2 = \min\Vert f \Vert_i^2$ ranging over all desired boundary functions $f$. Moreover, the minimum value is attained on the function $f_0\circ\varphi$ which is $0, 1$ on $\alpha_1$, $\beta_1$ and whose normal derivative vanishes on the complementary arcs (see \cite{BA}), that is, $\frac{1}{2\pi}m(Q) = \Vert f_0 \Vert_i^2$. Similarly, we have
$$
\frac{1}{2\pi}m(Q') = \min\Vert \tilde f\circ\tilde\varphi \Vert_i^2 = \min D(P(\tilde f\circ\tilde\varphi)\circ\tilde\varphi^{-1}\circ\iota ) = \min\Vert f \Vert_e^2 \leq \Vert f_0 \Vert_e^2.
$$
By the isomorphism of the identity operator $\mathcal H(\Gamma, \Omega) \to \mathcal H(\Gamma, \Omega_e)$, we have that $\Vert f_0 \Vert_e^2 \approx \Vert f_0 \Vert_i^2$, and thus $m(Q') \lesssim m(Q)$. The above reasoning clearly implies the other inequality $m(Q) \lesssim m(Q')$ by exchanging the roles of $Q$ and $Q'$. Then from $m(Q') \approx m(Q)$  it follows that $h$ is quasisymmetric, and thus $\Gamma$ is a quasicircle. 
\end{proof}

Theorem \ref{quasi} implies that $\Gamma$ is a quasicircle if and only if $\mathcal{H}(\Gamma, \Omega) = \mathcal{H}(\Gamma, \Omega_e)$. Now we define a ``two-sided" space $\mathcal{H}(\Gamma)$ to be $\mathcal{H}(\Gamma) =  \mathcal{H}(\Gamma, \Omega) = \mathcal{H}(\Gamma, \Omega_e)$. This is regarded as a Banach space consisting of all functions $f$ whose harmonic extensions $F$ to $\Omega$ have finite Dirichlet energies, where $f$ is assigned a norm $\Vert f \Vert_{\mathcal{H}(\Gamma)} := \sqrt{D(F)}$ by ignoring the difference of complex constant functions. Moreover, $\Vert f\circ\varphi\Vert_{H^{1/2}(\mathbb S)}^2$ is, by the Douglas formula, equal to $D(F)$, so we also have  
\begin{equation}\label{norm}
    \Vert f \Vert_{\mathcal{H}(\Gamma)}  = \Vert f\circ\varphi\Vert_{H^{1/2}(\mathbb S)}.
\end{equation}

We end this section with three comments about its content:
\begin{enumerate}
\item[(i)] 
As we have learned from one of the referees,  Theorem \ref{quasi} has been proved by Schippers-Staubach \cite{sc1} (also see \cite{sc2}) with the use of Proposition \ref{compo}, even without the assumption of rectifiability,   and the space $\mathcal{H}(\Gamma)$ has already been constructed by them. 
For the sake of completeness,  we  give a new proof of Theorem \ref{quasi} above taking a more direct approach for the rectifiable case. 
    \item[(ii)] 
    Let $\varphi$ be a quasiconformal mapping of $\mathbb C$  and $\Omega = \varphi(\mathbb D)$.  Gol'dshtein et al. (\cite{gol}, see also \cite{Jo}) have noticed that if $F\in \mathcal{D}(\Omega)$ then $F$ extended by 
\begin{equation}\label{Go}
    z\mapsto F\circ\varphi\Bigg(\frac{1}{\overline{\varphi^{-1}(z)}}\Bigg)
\end{equation}
 belongs to the homogeneous Sobolev space 
 $\dot W^{1,2}(\mathbb{C})$. 
 If now $g\in C_0^\infty(\mathbb C)$, the set of infinitely smooth functions in $\mathbb C$ with compact support,  we clearly have
 $$ \iint_\Omega|\nabla g|^2(z) dxdy\leq \iint_\mathbb{C}|\nabla g|^2(z) dxdy < +\infty,$$ and in particular $g|_\Gamma\in \mathcal{H}(\Gamma)$ by Dirichlet's  principle. Let $f\in \mathcal{H}(\Gamma)$ and $F$ its extension to $\mathbb{C}$ of the form \eqref{Go} as a function in $\dot W^{1,2}(\mathbb{C})$. There is a sequence of functions $F_n$ in  $C_0^\infty(\mathbb C)$  converging to $F$ in $\dot W^{1,2}(\mathbb{C})$ (see e.g. Chapter 11 of \cite{Leo}); using Dirichlet's principle again, we may conclude that the space of restrictions to $\Gamma$ of $C_0^\infty(\mathbb C)$  is dense in $\mathcal{H}(\Gamma)$.
 \item[(iii)] 
 Suppose $\varphi$ maps a quasidisk $\Omega'$ onto another one $\Omega$ conformally. It is easy to see that $f \mapsto f\circ\varphi$ is an isometric isomorphism from $\mathcal{H}(\Gamma)$ onto $\mathcal{H}(\Gamma')$. We now claim that if $\varphi$ is $k$-quasiconformal then $f \mapsto f\circ\varphi$ is still  a bounded   isomorphism from $\mathcal{H}(\Gamma)$ onto $\mathcal{H}(\Gamma')$. It is a consequence of quasiconformality and the change of variable formula. Specifically, if $F$ is the harmonic extension to $\Omega$ of $f \in \mathcal{H}(\Gamma)$ as in \eqref{Ff} then $D(F\circ\varphi)\lesssim D(F)$, where the implied constant only depends on $k$. By  \eqref{Diri} and Theoem \ref{quasi}, we conclude that 
 $$\Vert f\circ\varphi\Vert_{\mathcal{H}(\Gamma')}^2 \leq D(F\circ\varphi)\lesssim  D(F) = \Vert f\Vert_{\mathcal{H}(\Gamma)}^2.$$ 
 Combined with the quasiconformality of $\varphi^{-1}$, the above argument actually   implies the double inequality $\Vert f\circ\varphi\Vert_{\mathcal{H}(\Gamma')} \approx \Vert f\Vert_{\mathcal{H}(\Gamma)}$.
\end{enumerate}

\section{Chord-arc curves}
In this section we establish the equivalence between the Banach spaces $\mathcal{H}(\Gamma)$ and $H^{1/2}(\Gamma)$ when the curve $\Gamma$ is chord-arc. Let us start with a geometric description of  chord-arc curves. 
\begin{df}
A rectifiable Jordan curve $\Gamma$ is said to be a chord-arc curve (or $K$-chord-arc curve) if there exists a (least) positive constant $K$, called the chord-arc constant, such that 
$$
\mathrm{length} (\Gamma(z_1, z_2)) \leq K |z_1 - z_2|, \quad \text{for all}\; z_1, z_2 \in \Gamma
$$
where $\Gamma (z_1, z_2)$ is the shorter arc of $\Gamma$ between $z_1$ and $z_2$.
\end{df}
\noindent Chord-arc curves are also called ``Lavrentiev curves".  The inner domain of a chord-arc curve is called a chord-arc domain. A chord-arc curve is  the image of a circle under a bi-Lipschitz homeomorphism of $\mathbb C$. A sense-preserving bi-Lipschitz map of $\mathbb C$ onto $\mathbb C$ is quasiconformal but not vice versa. 
Hence, a chord-arc curve  must be a quasicircle. 
 Indeed, it is exactly a quasicircle with `` regularity property" that will be explored in detail in the next section. 
 In this section, We assume without loss of generality that $\text{length}(\Gamma) = 2\pi$. 

 Now we can state the main result of this section. 
\begin{theo}\label{appro}
The identity operator $\mathcal{H}(\Gamma) \to H^{1/2}(\Gamma)$ is a bounded isomorphism with respect to $\Vert\cdot\Vert_{\mathcal{H}(\Gamma)}$ and $\Vert\cdot\Vert_{H^{1/2}(\Gamma)}$  if the curve  $\Gamma$ is chord-arc.
\end{theo}

Before proceeding the proof of Theorem \ref{appro}, we point out the following basic observation.
 
 \begin{lem}\label{sin}
Let $\Gamma$ be a rectifiable Jordan curve  with $\text{length}(\Gamma) = 2\pi$. Set $z(s)$, $0 \leq s < 2\pi$, to be an arc-length parametrization of $\Gamma$. Then it holds that, for any $s,t\in[0,2\pi)$,
\begin{equation}
\frac{\pi}{2}\vert e^{it} - e^{is}\vert = \pi\left\vert\sin{\frac{t-s}{2}}\right\vert \ge |z(t) - z(s)|. \label{ineqref1}
\end{equation}
Moreover, if $\Gamma$ is $K$-chord-arc, then
\begin{equation}\label{ineqref2}
  \frac{1}{K} \vert e^{it} - e^{is}\vert = \frac{2}{K} \left\vert\sin{\frac{t-s}{2}}\right\vert \leq |z(t) - z(s)|.
\end{equation}
 \end{lem}
 
\begin{proof} 
If $|t - s|\leq\pi$, it is easy to see that 
$$
\Big\vert\sin{\frac{t-s}{2}}\Big\vert \geq \frac{2}{\pi}\Big\vert\frac{t-s}{2}\Big\vert \geq \frac{1}{\pi}|z(t) - z(s)|.
$$
If $\pi < |t - s| < 2\pi$, we may also see that
$$
\Big\vert\sin{\frac{t-s}{2}}\Big\vert = \sin{\frac{2\pi - |t - s|}{2}} \geq \frac{2}{\pi}\cdot\frac{2\pi - |t - s|}{2} \geq \frac{1}{\pi}|z(t) - z(s)|.
$$
This proves the inequality \eqref{ineqref1}. 

Further, with the use of the chord-arc condition one may see that if $|t - s|\leq\pi$, 
$$
\Big\vert\sin{\frac{t-s}{2}}\Big\vert \leq \frac{1}{2}|t - s| \leq \frac{K}{2}|z(t) - z(s)|,
$$
and if $\pi < |t - s| < 2\pi$, 
$$
\Big\vert\sin{\frac{t-s}{2}}\Big\vert \leq \frac{1}{2}(2\pi - |t - s|) \leq \frac{K}{2}|z(t) - z(s)|.
$$
This gives a proof of the inequality \eqref{ineqref2}. 
\end{proof}

\begin{proof}[Proof of Theorem \ref{appro}]
    Suppose that $\Gamma$ is a $K$-chord-arc curve with $\text{length}(\Gamma) = 2\pi$. We notice by Lemma \ref{sin} that its arc-length parametrization $z(e^{is})$, $0 \leq s < 2\pi$, satisfies  
 \begin{equation}\label{double}
   \frac{1}{K}\vert e^{it} - e^{is}\vert \leq \vert z(e^{it}) - z(e^{is})\vert \leq \frac{\pi}{2}\vert e^{it} - e^{is}\vert
 \end{equation}   
for any $s,t\in[0,2\pi)$.  
Then, we see that $z$ is a bi-Lipschitz embedding of $\mathbb S$ into $\mathbb C$. 
Recall that $\varphi$ is a conformal map of $\mathbb D$ onto the chord-arc domain $\Omega$, and a homeomorphism of closures $\mathbb D\cup\mathbb S$ onto $\Omega\cup\Gamma$. It follows that $\varphi$ restricted to $\mathbb S$ is a quasisymmetric mapping of $\mathbb S$ onto $\Gamma$. As a consequence, $z^{-1}\circ\varphi$ is a quasisymmetry of $\mathbb S$. 
Note that 
$$
\|f\|_{H^{1/2}(\Gamma)}^2  = \frac{1}{4\pi^2}\int_{0}^{2\pi}\int_{0}^{2\pi}\Big\vert \frac{f(z(e^{it})) - f(z(e^{is}))}{z(e^{it}) - z(e^{is})} \Big\vert^2 dtds
$$
and 
$$
\|f\circ z\|_{H^{1/2}(\mathbb S)}^2 =\frac{1}{4\pi^2} \int_{0}^{2\pi}\int_{0}^{2\pi}\Big\vert \frac{f(z(e^{it})) - f(z(e^{is}))}{e^{it} - e^{is}} \Big\vert^2 dtds.
$$
Using \eqref{double} gives 
$$ \frac{4}{\pi^2}\|f\circ z\|_{H^{1/2}(\mathbb S)}^2 \leq \|f\|_{H^{1/2}(\Gamma)}^2 \leq K^2 \|f\circ z\|_{H^{1/2}(\mathbb S)}^2.$$
By Proposition \ref{compo}, the quasisymmetry of $z^{-1}\circ\varphi$ implies 
$$
\|f\circ \varphi\|_{H^{1/2}(\mathbb S)} = \|f\circ z \circ (z^{-1}\circ\varphi)\|_{H^{1/2}(\mathbb S)}\approx \|f\circ z\|_{H^{1/2}(\mathbb S)}.
$$
Then, we conclude that 
\begin{equation}\label{h}
    \|f\circ \varphi\|_{H^{1/2}(\mathbb S)} \approx \|f\|_{H^{1/2}(\Gamma)}. 
\end{equation}
It follows from \eqref{norm} that  $\|f\|_{\mathcal{H}(\Gamma)} \approx \|f\|_{H^{1/2}(\Gamma)}.$ This completes the proof of Theorem \ref{appro}. 
\end{proof}

We noticed recently that the corresponding result to Theorem \ref{appro} for the critical Besov space has been obtained in \cite{LS} via different reasoning. 
By checking this proof we can see that \eqref{h} is still valid when replacing the conformal mapping $\varphi$ by any quasisymmetric mapping $h$ of $\mathbb S$ onto $\Gamma$, and thus we have the following
\begin{cor}\label{Cor1}
    Let $h$ be a quasisymmetric mapping of $\mathbb S$ onto a chord-arc curve $\Gamma$. Then the composition operator $V_{h}: f \mapsto f\circ h$ gives a bounded isomorphism of $H^{1/2}(\Gamma)$ onto $H^{1/2}(\mathbb S)$.
\end{cor}

\section{About the necessity of the chord-arc condition}
 In the last section (i.e. Theorem \ref{appro}) we gave a proof of the fact that if $\Gamma$ is a chord-arc curve then $\mathcal{H}(\Gamma)=H^{1/2}(\Gamma)$. In other words we have proven the existence of a constant $C>0$ such that the two following inequalities hold for chord-arc curves:
 \begin{itemize}
\item[(a)]$\|\cdot\|_{H^{1/2}(\Gamma)}\le C\|\cdot\|_{\mathcal{H}(\Gamma)}$.
\item[(b)]$\|\cdot\|_{\mathcal{H}(\Gamma)}\le C\|\cdot\|_{H^{1/2}(\Gamma)} $.
\end{itemize}
Inequality (a) is equivalent to $\mathcal{H}(\Gamma)\subset H^{1/2}(\Gamma)$ and Inequality (b) to $H^{1/2}(\Gamma)\subset\mathcal{H}(\Gamma)$.
The purpose of this section is to examine the possible converse to Theorem \ref{appro}; that is the following theorem.
\begin{theo} \label{rec} Let $\Gamma$ be a rectifiable quasicircle. The following two statements hold.
\begin{itemize}
    \item[{\rm(1)}] If  \rm{(a)} holds then $\Gamma$ is a chord-arc curve.
    \item[{\rm(2)}] There exists a non-chord-arc rectifiable quasicircle $\Gamma$ such that \rm{(b)} holds. 
\end{itemize}
\end{theo}

We  will discuss notions of Ahlfors-regularity in the next subsection; that will be the key tool in the proof of part $(1)$ of Theorem \ref{rec}  and also have independent interests of their own. 
We then prove parts $(1)$ and $(2)$ of Theorem \ref{rec} in the final two subsections.

\subsection{Ahlfors-regular curves}
In this section we discuss about Ahlfors-regular curves, sometimes also named Ahlfors-David regular curves. Ahlfors' name appears because this author already considered this condition in \cite{Ahl} and David's one because this author proved that these curves are precisely the curves for which the Cauchy operator is bounded on $L^2$  (see \cite{Dav}). In the sequel we will simply call regular these curves whose precise definition is
\begin{df} 
Let $\Gamma$ be a rectifiable curve in the plane. We say that $\Gamma$ is M-regular if for 
any $z \in \mathbb C$ and $r > 0$,
$$\mathrm{length}(\Gamma\cap D(z,r))\leq Mr,$$
where $D(z, r)$ stands for the open disk centered at $z$ and of radius $r$. 
\end{df}
Our aim is to prove the following theorem giving two equivalent properties of regularity. The first of these properties involves the Riemann sphere $S^2$, which is $\mathbb C\cup \{\infty\}$ equipped with the metric
$$ d\rho=\frac{|dz| }{1+|z|^2}.$$
We will denote by s-length the spherical length of a curve.  
The group of conformal automorphisms of $S^2$ is the group of M\"obius transformations
$$ z\mapsto \frac{az+b}{cz+d},\quad ad-bc=1,$$
which is thus isomorphic to ${\rm PSL}(2, \mathbb C)$. 

\begin{lem} \label{Mob}
If $\Gamma$ is an $M$-regular curve then $T(\Gamma)$ is, for any M\"obius transformation $T$, $12M$-regular.  
\end{lem}
\begin{proof}
   Every M\"obius transformation $T$ may be factorized as $T=S_1\circ \Theta\circ S_2 $ where $S_i,\,i=1,2$, are similitudes, and $\Theta(z)=1/z$. The property of regularity, along with the constant $M$, is obviously invariant for similitudes so it suffices to prove the lemma for $\Theta$.
    
Let us consider an $M$-regular curve $\Gamma$. Set $z_0=r_0e^{it}\in \mathbb{C},\,r>0,\,D =D (z_0,r)$. Without loss of generality we may assume that $z_0$ is real positive (i.e., $t=0$). Set $\Delta  = \Theta^{-1}(D)$. If $r<r_0/2$ then $\Delta$ is the disk centered at $\frac{r_0}{r_0^2-r^2}$ with radius $\frac{r}{r_0^2-r^2} < \frac{4r}{3r_0^2}$. Moreover if $z\in\Delta,\,|z| \ge \frac{1}{r_0+r} > \frac{2}{3r_0}$. It follows that 
\begin{align*}
 \mathrm{length}(\Theta(\Gamma)\cap D)&=\int_{\Gamma\cap\Delta}\frac{|dz|}{|z|^2} 
 \leq \int_{\Gamma\cap\Delta}\frac{|dz|}{\left(2/(3r_0)\right)^2}\\
 &\leq M\left(\frac{4r}{3r_0^2}\right)\left(\frac{3r_0}{2}\right)^2= 3Mr.
 \end{align*}
If now  $r\ge r_0/2$, $D$ is included in $D(0,3r)$ and $\Delta$ in $\{|z|>1/(3r)\}$. We may then write
\begin{align*}
\mathrm{length}(\Theta(\Gamma)\cap D)&=\int_{\Gamma \cap \Delta}\frac{|dz|}{|z|^2}\le \int_{\Gamma \cap \{|z|>\frac{1}{3r}\}}\frac{|dz|}{|z|^2}\\
&= \sum\limits_{n\ge 0}\int_{\Gamma\cap\{\frac{2^n}{3r}<|z|\le \frac{2^{n+1}}{3r}\}}\frac{|dz|}{|z|^2}\le\sum\limits_{n\ge 0}M\frac {2^{n+1}}{3r}\frac{9r^2}{2^{2n}} = 12Mr.
\end{align*}
\end{proof}

We are now ready to state the main result of this subsection.

\begin{theo}\label{mey}
Let $\Gamma$ be a rectifiable curve in the plane. Then the following statements are equivalent.
\begin{itemize}
    \item[{\rm(i)}] $\Gamma$ is an $M$-regular curve,
    \item[{\rm(ii)}] There exists $K>0$ such that for any M\"obius transformation $T$,
$$ \text{\rm s-length}(T(\Gamma))\leq K,$$
    \item[{\rm(iii)}] There exists $C>0$ such that for every $w\notin \Gamma$,
$$\mathrm{length}(M_w(\Gamma))\leq \frac{C}{d(w,\Gamma)},$$
where $M_w(z) = 1/(z - w)$ is an inversion and 
$d(w,\Gamma)$ is the distance from $w$ to $\Gamma$.
\end{itemize}
\end{theo}

\begin{proof} 
We first show that (i) $\Rightarrow$ (ii). Suppose $\Gamma$ is $M$-regular. It follows from  Lemma \ref{Mob} that for any M\"obius transformation $T$, $T(\Gamma)$ is $12M$-regular. 
Let $\eta:\,[0,1]\to \mathbb C$ be an absolutely continuous parametrization of $T(\Gamma)$.  We see the s-length of $T(\Gamma)$ is
$$
\int_0^1\frac{|\eta'(t)|dt}{1+|\eta(t)|^2}=I+\sum\limits_{n=0}^\infty I_n,$$
where 
$$ I=\int_{|\eta(t)|<1}\frac{|\eta'(t)|dt}{1+|\eta(t)|^2}\leq 12M$$
while 
$$ I_n=\int_{2^n\leq |\eta(t)|<2^{n+1}}\frac{|\eta'(t)|dt}{1+|\eta(t)|^2}\leq \frac{12M2^{n+1}}{1+2^{2n}}\leq 12M2^{1-n},$$
and thus we have $\text{\rm s-length}(T(\Gamma))\leq K$ by taking $K = 60M$.

Now we prove (ii) $\Rightarrow$ (iii). Let $\gamma:\,[0,1]\to \mathbb C$ be an absolutely continuous parametrization of $\Gamma$. 
Set $w\notin \Gamma$ so that
$$\mathrm{length}(M_w(\Gamma))=\int_0^1\frac{|\gamma'(t)|dt}{|\gamma(t)-w|^2}.$$ We may write, by definition of $d:=d(w,\Gamma),$
$$|\gamma(t)-w|^2\geq\frac 12(|\gamma(t)-w|^2+d^2)$$
for all $t\in [0, 1]$, 
so that, if we define $\eta(t)=\frac{\gamma(t)-w}{d},$
$$\int_0^1\frac{|\gamma'(t)|dt}{|\gamma(t)-w|^2}\leq 2\int_0^1\frac{|\gamma'(t)|dt}{d^2+|\gamma(t)-w|^2} = \frac 2d\int_0^1\frac{|\eta'(t)|dt}{1+|\eta(t)|^2}\leq \frac{2K}{d}$$
by (ii). 

 Finally, we prove (iii) $\Rightarrow$ (i).  This is the hardest part of the proof. We first notice that in order to test the ``regularity property" of a curve we may replace disks by squares of diameter not greater than the diameter of $\Gamma$, moreover centered on $\Gamma$. Let $\mathcal{C}$ be such a square. By the invariance of regularity under similitudes, one may assume that $\mathcal{C}$ has side-length $1$ and center $0$. We define
$$ L = \mathrm{length}(\mathcal{C}\cap\Gamma).$$
The goal is to estimate $L$ from above by a constant depending only on $C$, the constant in  (iii).
If $L \leq 100 C$, we are done. 
In the other case let $n$ be the integer part of $\frac{L}{2C}$. It implies that $L < 3Cn$. It remains to prove that $n$ is bounded from above by a constant depending only on $C$. We will do this in two steps.

We first cut $\mathcal{C}$ into $n^2$ sub-squares of side-length $\frac 1n$, that we call $\mathcal{C}_j$. We denote by $k\mathcal{C}_j$ ( $k>0$) the intersection of $\mathcal{C}$ and the square with the same center as $\mathcal{C}_j$  whose side-length is $k$ times the side-length of $\mathcal{C}_j$. 
If $z_0\in\mathcal{C}\!\setminus\!\!\Gamma$ then we have on one side
$$\int_{\Gamma\cap\mathcal{C}}|z-z_0|^{-2}|dz|\geq \frac{L}{2},$$
while, on the other side, by (iii),
$$\int_{\Gamma\cap\mathcal{C}}|z-z_0|^{-2}|dz|\leq \frac{C}{d(z_0,\Gamma)}.$$
Consequently,
\begin{equation}\label{d}
    d(z_0,\Gamma)\leq \frac{2C}{L} \leq \frac{1}{n},
\end{equation}
from which it follows that every $2\mathcal{C}_j$ meets $\Gamma$. 
Since diam$(\mathcal{C}_j)\leq$diam$(\Gamma)/10$, $\Gamma$ also meets the complement of $3\mathcal{C}_j$ and, as a consequence,
$$\mathrm{length}(\Gamma\cap3\mathcal{C}_j)\geq \frac 1n.$$

Next, note that
$$\int_\Gamma |z-z_0|^{-2}|dz|\geq \frac 19\sum\limits_{j\in J}\int_{3\mathcal{C}_j\cap\Gamma}|z-z_0|^{-2}|dz|,$$
where $J=\{j: z_0\notin 27\mathcal{C}_j\}$. If $z,z'\in 3\mathcal{C}_j,\, j\in J,$ then $|z'-z_0|\leq 2|z-z_0|$, from which it follows easily that
\begin{align*}
    \iint_{3\mathcal{C}_j}|z-z_0|^{-2}dxdy&\leq n\times \mathrm{length}(\Gamma\cap3\mathcal{C}_j) \times \Big(\frac{3}{n}\Big)^2\times \max_{z \in 3\mathcal{C}_j}|z - z_0|^{-2} \\
    &\leq \frac{36}{n}\int_{\Gamma\cap3\mathcal{C}_j}|z-z_0|^{-2}|dz|.
\end{align*}
By comparison with an integral we have
\begin{align*}
\sum\limits_{j\in J}\iint_{3\mathcal{C}_j}|z-z_0|^{-2}dxdy&\geq\int_0^{\pi/4}\int_{12/n}^{1/2}\frac 1r drd\theta\\
&\geq\frac{\pi}{4}\log \frac{n}{24}.
\end{align*}
Combining these estimates we deduce that 
$$\int_\Gamma |z-z_0|^{-2}|dz|\geq \frac{n}{500}\log\frac{n}{24},$$
and this, combined with (iii), implies 
\begin{equation}\label{dd}
    d(z_0,\Gamma)\leq  \frac{500C}{n\log\frac{n}{24}}.
\end{equation}
Let $n'$ be the integer part of $\frac{n\log \frac{n}{24}}{2000C}$. We can suppose that $n'$ is bigger than $n$. If not, $n$ must be bounded from above by a constant depending only on $C$, and then we are done. 
 By \eqref{dd} we have
\begin{equation}
    d(z_0,\Gamma)\leq \frac{1}{4n'}
\end{equation}
which is an improvement of \eqref{d}. Then we repeat the above discussion. 
 We  cut $\mathcal{C}$ into $n'^2$ sub-squares of side-length $\frac{1}{n'}$, and we realize that each of these sub-squares meets $\Gamma$. On the other hand, $\Gamma$  meets the complement of $3\mathcal{C}_j$. Then, we see
 $$\mathrm{length}(\Gamma\cap3\mathcal{C}_j)\geq \frac{1}{n'}.$$
 As a consequence,  
\begin{equation}\label{MM}
    L > \frac{1}{9} \times \sum_{j=1}^{n'^2} \mathrm{length}(\Gamma\cap3\mathcal{C}_j)        > \frac{n'}{9} > \frac{1}{10}\times \frac{n\log\frac{n}{24}}{2000C}.
\end{equation}
Combining this with $L < 3Cn$ we conclude that $n < 24e^{60000C^2}$. Finally, we deduce that 
$L < 72Ce^{60000C^2},$ where $C$ is the constant occurring in (iii). 
The proof of the theorem is complete.
\end{proof}

\subsection{Necessity of the chord-arc property: part $(1)$ of Theorem \ref{rec}}
In this subsection we prove part  $(1)$ of Theorem \ref{rec}.
To this end we first consider  $w\in \Omega_e$ and apply the Inequality (a) to the analytic function $f(z)=\frac{1}{z-w}$ defined on $\Omega\cup\Gamma$. The square of the integral on the right hand side is then
\begin{align*}
   \|f\|_{\mathcal{H}(\Gamma)}^2  & = \frac{1}{\pi}\iint_{\Omega}|f'(z)|^2 dxdy= 
\frac{1}{\pi}\iint_\Omega |z - w|^{-4} dxdy\\
&\leq \frac{1}{\pi}\iint_{|z - w| > d(w,\Gamma)}|z - w|^{-4}dx dy \\
&= d(w,\Gamma)^{-2}
\end{align*}
as we see using polar coordinates. 
The square of the integral on the left hand side is
\begin{align*}
\|f\|_{H^{1/2}(\Gamma)}^2 &= \frac{1}{4\pi^2}\iint_{\Gamma\times\Gamma}\left| \frac{f(z)-f(\zeta)}{z-\zeta}\right|^2 |dz||d\zeta| \\
&= \frac{1}{4\pi^2}\iint_{\Gamma\times\Gamma}\left|\frac{\frac{1}{z-w}-\frac{1}{\zeta-w}}{z-\zeta}\right|^2|dz||d\zeta|\\ 
&=\frac{1}{4\pi^2}\int_\Gamma\int_{\Gamma}\frac{1}{|z-w|^2|\zeta-w|^2}|dz||d\zeta|\\
&=\left(\frac{1}{2\pi}\mathrm{length} \left(M_w(\Gamma)\right)\right)^2
\end{align*}
by Fubini theorem and with the notations of last subsection. 

We have thus proven that if (a) is valid then there exists a constant $C > 0$ such that 
\begin{equation}\label{ineq}
    \mathrm{length}(M_w(\Gamma))\leq \frac{C}{d(w,\Gamma)}
\end{equation}
for any $w\in \Omega_e$. Next we consider $w\in \Omega$ and apply $(a)$ to the analytic function $f(z) = \frac{1}{z-w}$ defined on $\Omega_e\cup\Gamma$. 
Using the fact that $\mathcal{H}(\Gamma)$ is a ``two-sided" space for a quasicircle $\Gamma$, we may  similarly see that the inequality \eqref{ineq} 
remains true for any $w\in \Omega$. We can then invoke Theorem \ref{mey}, and see that $\Gamma$ has to be regular. Consequently, $\Gamma$ is chord-arc since we have assumed it is a quasicircle. This completes the proof of part $(1)$ of Theorem \ref{rec}. 

As noted above, it follows from Proposition \ref{compo} that when $h$ is a quasisymmetric mapping from $\mathbb S$ onto a rectifiable quasicircle $\Gamma$ it holds that $\|f\|_{\mathcal{H}(\Gamma)} = \|f\circ\varphi\|_{H^{1/2}(\mathbb S)} \approx \|f\circ h\|_{H^{1/2}(\mathbb S)}$, and thus  part $(1)$ of Theorem \ref{rec} does immediately imply the following fact.
\begin{cor}\label{Cor2}
   Let $h$ be a quasisymmetric mapping of $\mathbb S$ onto a rectifiable quasicircle $\Gamma$. Let the composition operator $(V_{h})^{-1}: f \mapsto f\circ h^{-1}$ be a bounded linear operator from $H^{1/2}(\mathbb S)$ into $H^{1/2}(\Gamma)$. Then $\Gamma$ is a chord-arc curve.
\end{cor}

\subsection{Non-Smirnov domains: part $(2)$ of Theorem \ref{rec}}
In this subsection we present a counter-example to show part $(2)$ of Theorem \ref{rec}. 
When $\Gamma$ is a Jordan curve with $\Omega$ as interior domain and $\varphi:\mathbb D\to \Omega $ a Riemann mapping, we know (F. and M. Riesz theorem \cite{pom}) that $\Gamma$ is rectifiable if and only if $\varphi'$ belongs to the Hardy space $H^1(\mathbb D)$.
\begin{df} 
Let $\Gamma$ be a rectifiable curve in the plane. We say that $\Omega$ is a Smirnov domain if $\varphi'$ is an outer function of $H^1(\mathbb D)$; that is, 
\begin{equation}\label{outer}
    \log\vert\varphi'(z)\vert = \int_{\mathbb S} p(z,\zeta) \log\vert\varphi'(\zeta)\vert|d\zeta|\quad \text{\rm for}\; z \in \mathbb D,
\end{equation}
where $p$ is the Poisson kernel.
\end{df}
It has been shown by Lavrentiev \cite{lav} that chord-arc domains are Smirnov domains and later by the second author \cite{zin} that Jordan domains with Ahlfors-regular boundary also have Smirnov property. On the other hand there exists a rectifiable quasidisk $\Omega$ whose Riemann map $\varphi$ 
satisfies that $\varphi'$ is an inner function, see \cite{dss, kah}; that is, 
\begin{equation}\label{ex}
    \vert\varphi'(\zeta)\vert = 1 \quad \text{\rm for almost\;all }\, \zeta \in \mathbb S, \quad \vert\varphi'(z)\vert < 1 \quad \text{\rm for}\, z \in \mathbb D. 
\end{equation}
In particular, harmonic measure on $\Gamma$ with respect to  $\varphi(0)$ is equal to arc-length measure on $\Gamma$ despite the fact that $\Omega $ is not a disk. 
But the Smirnov condition \eqref{outer} is not satisfied because of \eqref{ex}. 

 We are going to exploit this fact in order to show that this domain satisfies Inequality (b) even if it is not chord-arc. This follows immediately from the following observation 
\begin{align*}
\|f\|^2_{H^{1/2}(\Gamma)}&=\int_0^{2\pi}\int_0^{2\pi}\left|\frac{f(\varphi(e^{it}))-f(\varphi(e^{is}))}{\varphi(e^{it})-\varphi(e^{is})}\right|^2\frac{dt}{2\pi} \frac{ds}{2\pi}    \\
& \geq \frac{4}{\pi^2} \int_{0}^{2\pi}\int_{0}^{2\pi}\Big\vert \frac{f(\varphi(e^{it})) - f(\varphi(e^{is}))}{e^{it} - e^{is}} \Big\vert^2 \frac{dt}{2\pi} \frac{ds}{2\pi}\\
&=\frac{4}{\pi^2} \|f\|^2_{\mathcal{H}(\Gamma)}, \quad \text{for} \, f\in H^{1/2}(\Gamma).
\end{align*}
Here, the inequality ``$\geq$" is due to the inequality \eqref{ineqref1}. 
This completes the proof of part $(2)$ of Theorem \ref{rec}. 

Immediately, by taking the quasisymmetric mapping $h$ to be $\varphi$ on $\mathbb S$ we have the following
\begin{cor}\label{Cor3}
    There exists a quasisymmetric mapping $h$ of $\mathbb S$ onto a rectifiable quasicircle $\Gamma$ such that the composition operator $V_{h}: f \mapsto f\circ h$ is a bounded linear operator from $H^{1/2}(\Gamma)$ into $H^{1/2}(\mathbb S)$, but $\Gamma$ is non-chord-arc. 
\end{cor}

As a side remark, in \cite{js}, it is proven that if $\Omega$ is a quasidisk such that $\varphi'$, the derivative of its Riemann map, is a singular inner function, then $\Omega_e$ has to be a Smirnov domain. It follows that there are counterexamples which are Smirnov domains.

\medskip

\noindent\textbf{Acknowledgement} We would like to  thank Yves Meyer for having shared his idea leading to (iii) $\Rightarrow$ (i) of Theorem \ref{mey}. We also thank the referees for many constructive comments, which helped us to greatly improve the quality of this paper.

\section*{Declarations}
\noindent\textbf{Conflict of interest.} On behalf of all authors, the corresponding author states that there is no conflict of interest regarding the publication of this paper.

\bibliographystyle{alpha}

\end{document}